\pdfoutput=1
\documentclass{article}

\usepackage{amsmath,amssymb}
\usepackage{graphicx}
\usepackage{theorem}\theorembodyfont{\normalfont}

\usepackage{tikz}
\usetikzlibrary{calc,arrows.meta,positioning,decorations.pathmorphing}
\usepackage{combgames}

\newtheorem{definition}{Definition}[section]
\newtheorem{lemma}[definition]{Lemma}
\newtheorem{theorem}[definition]{Theorem}
\newtheorem{corollary}[definition]{Corollary}
\newtheorem{proposition}[definition]{Proposition}

\newtheorem{proof}{Proof\hspace*{-2pt}}

\def\QED{\hspace*{\fill}$\square$}

\newif\ifDEBUG

\newcommand{\imply}{\Rightarrow}
\newcommand{\iffsmall}{\hspace*{-5pt}\Leftrightarrow\hspace*{-5pt}}

\newcommand{\cgpos}[2]{\combgame{\{#1 \,|\, #2\}}}

\newcommand{\Z}{\mathbb{Z}}
\newcommand{\ZZ}{\cgpos{\Z}{\Z}}
\newcommand{\D}{\mathbb{D}}
\newcommand{\DD}{\cgpos{\D}{\D}}

\newcommand{\LS}[1]{\mathrm{L}_{#1}}
\newcommand{\RS}[1]{\mathrm{R}_{#1}}

\newcommand{\gd}{\mathrm{gd}}

\newcommand{\Reduction}[9]{\mathrel{%
  {}_{#1}^{#2}%
  \raise#7\hbox{%
    \vtop{\ialign{##\crcr%
      \hfil\raise#8\hbox{#4}\hfil\crcr%
      \noalign{\nointerlineskip}%
      #9\crcr%
      \noalign{\nointerlineskip}%
      \hfil#3\hfil\crcr%
    }}%
  }%
  {}_{\scriptscriptstyle #5}^{#6}}}%
\newcommand{\move}[2]{%
  \Reduction{}{}%
    {\hspace{3pt}$\scriptscriptstyle{#1}$\hspace{6pt}}{\hspace{3pt}$\scriptstyle{#2}$\hspace{6pt}}%
    {}{}{5pt}{1pt}{\rightarrowfill}}

\newcommand{\lmove}[2]{%
  \Reduction{}{}%
    {\hspace{6pt}$\scriptscriptstyle{#1}$\hspace{3pt}}{\hspace{6pt}$\scriptstyle{#2}$\hspace{3pt}}%
    {}{}{5pt}{1pt}{\leftarrowfill}}

\begin{document}

\noindent
\begin{center}
{\LARGE \bf
  Various Diamond Properties\\[4pt]
  in Combinatorial Game Theory}\\[16pt]
{\large \bf
  KUSAKARI Keiichirou\footnotemark[1]\hspace{20pt}
  ABUKU Tomoaki\footnotemark[1]}\\[8pt]
{\large December 25, 2025}
\end{center}
\vspace*{4pt}
\footnotetext[1]{Gifu University}


\begin{center}
\begin{minipage}{0.8\textwidth}
\textbf{SUMMARY}~~
We investigate conditions under which positions in combinatorial games admit simple values.
We introduce a unified diamond framework, the $\Diamond_\mathcal{A}$-property ($\mathcal{A}\in\{\Z,\D\}$), for sets of positions closed under options.
Under certain conditions, this framework guarantees that all values are integers,
dyadic rationals, or pairs $\cgpos{m}{n}$ (on $\Z$ or $\D$).
As an application, we establish that every position in \textsc{Yashima}
game on bipartite graphs has an integer pair value.
\end{minipage}
\end{center}


\section{Introduction}

A combinatorial game is a two-player game
with no chance elements and no hidden information,
and the collection of results concerning such games forms the field of
combinatorial game theory \cite{ANW19,S13}.
In this theory, each game position is assigned a position value,
and these values can be handled algebraically
to support various kinds of analysis.
Position values include the set $\D$ of numbers (the dyadic rationals),
as well as non-numeric values such as
$\cgstar$ $(\cong \cgpos{0}{0})$ and $\cgup$ $(\cong \cgpos{0}{\cgstar})$.
However, since the game tree often grows very fast,
the analysis tends to be hard and the resulting values can become quite complicated.

As an example, we introduce \textsc{Yashima} game \cite{A90}.
At the beginning, each of the two players places one token on a vertex of an undirected graph.
Following the usual convention, the players are referred to as Left and Right.
On each turn, a player moves their token along an edge to an adjacent vertex
and deletes the edge just traversed.
The players then continue alternately in this manner.
The two tokens may never occupy the same vertex simultaneously.
The player who has no legal move loses.

An example of \textsc{Yashima} game is shown in Fig.\ref{fig:yashima1}.
In this figure, 
\textcircled{\raisebox{.2pt}{\scriptsize L}} and
\textcircled{\raisebox{.2pt}{\scriptsize R}}
denote the tokens of Left and Right, respectively. 
In this example, Right eventually has no legal move and therefore loses.
\begin{figure}[htb]
  \centering
\begin{tikzpicture}[x=12pt,y=12pt,thick]
\node at (0,0) {
\begin{tikzpicture}
\node (p01) at (0,2) [draw,circle,inner sep=2pt] {L};
\node (p11) at (2,2) [draw,circle,inner sep=5pt] {};
\node (p00) at (0,0) [draw,circle,inner sep=5pt] {};
\node (p10) at (2,0) [draw,circle,inner sep=2pt] {R};
\draw [double distance=2pt] (p00) to (p10);
\draw (p10) to (p11);
\draw (p00) to (p01);
\end{tikzpicture}};
\node at (6,0) {
\begin{tikzpicture}
\node (p01) at (0,2) [draw,circle,inner sep=5pt] {};
\node (p11) at (2,2) [draw,circle,inner sep=5pt] {};
\node (p00) at (0,0) [draw,circle,inner sep=2pt] {L};
\node (p10) at (2,0) [draw,circle,inner sep=2pt] {R};
\draw [double distance=2pt] (p00) to (p10);
\draw (p10) to (p11);
\end{tikzpicture}};
\node at (12,0) {
\begin{tikzpicture}
\node (p01) at (0,2) [draw,circle,inner sep=5pt] {};
\node (p11) at (2,2) [draw,circle,inner sep=2pt] {R};
\node (p00) at (0,0) [draw,circle,inner sep=2pt] {L};
\node (p10) at (2,0) [draw,circle,inner sep=5pt] {};
\draw [double distance=2pt] (p00) to (p10);
\end{tikzpicture}};
\node at (18,0) {
\begin{tikzpicture}
\node (p01) at (0,2) [draw,circle,inner sep=5pt] {};
\node (p11) at (2,2) [draw,circle,inner sep=2pt] {R};
\node (p00) at (0,0) [draw,circle,inner sep=5pt] {};
\node (p10) at (2,0) [draw,circle,inner sep=2pt] {L};
\draw (p00) to (p10);
\end{tikzpicture}};
\draw [->] (2.2,0) to (3.8,0); \node at (3,-.5) {\small L};
\draw [->] (8.2,0) to (9.8,0); \node at (9,-.5) {\small R};
\draw [->] (14.2,0) to (15.8,0); \node at (15,-.5) {\small L};
\end{tikzpicture}
\caption{\textsc{Yashima} game}
\label{fig:yashima1}
\end{figure}
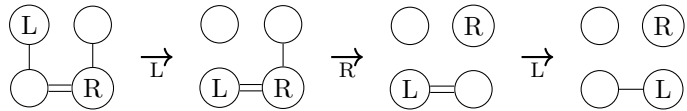

In our analysis of \textsc{Yashima} game,
we find that in many cases the position value takes a remarkably simple form,
namely an integer pair $\cgpos{m}{n}$.
Although this is not valid in full generality,
we discover that it holds for bipartite graphs.

As a key to this phenomenon,
we introduce the notion of the diamond property,
illustrated in Fig.\ref{fig:diamond}.
The name derives from the diamond-shaped diagram.
This property typically arises when the intended move is not blocked by the opponent, 
a situation that frequently occurs in the early stages of the game.
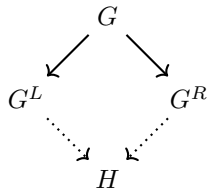
\begin{figure}[htb]
\centering
\begin{tikzpicture}[x=22pt,y=22pt,thick,->]
\node (o) at (-.8,0) {};
\node (g) at (1.4,2.8) {$G$};
\node (gl) at (0,1.4) {$G^L$};
\node (gr) at (2.8,1.4) {$G^R$};
\node (h) at (1.4,0) {$H$};
\draw (g) to (gl);
\draw (g) to (gr);
\draw [dotted] (gl) to (h);
\draw [dotted] (gr) to (h);
\end{tikzpicture}
  \caption{Diamond property}
  \label{fig:diamond}
\end{figure}

A related work considers a set of positions $\mathcal{G}$ that is closed under options.
For such a set, authors in \cite{CHNS21} proposed two properties:
(i) the $F_2$-property, which guarantees $\mathcal{G} \subset \Z$, and
(ii) the $F_1$-property, which guarantees $\mathcal{G} \subset \D$.
To maintain consistent terminology in this paper, we refer to
the $F_2$-property as the $\Diamond^{\le}$-property, and
the $F_1$-property as the $\Diamond_L^{\le}$ and $\Diamond_R^{\le}$-properties.
These results are very closely related to our property:
(iii) the $\Diamond$-property,
which guarantees $\mathcal{G} \subset \ZZ$ (the integer pairs);
in particular, the $F_2$-property subsumes the $\Diamond$-property as a special case.
These observations suggest the existence of a more general framework
that unifies all of them.

In Section~3, we introduce such a framework
called the $\Diamond_\mathcal{A}$-property for $\mathcal{A} \in \{\Z,\D\}$.
This provides a unified method guaranteeing that,
for any set $\mathcal{G}$ closed under options,
all positions in $\mathcal{G}$ have values that are integer pairs, integers, number pairs, or numbers.
In Section~4,
we demonstrate that both our $\Diamond$-based framework and the results
of \cite{CHNS21} emerge as particular instances of this approach.
We further present several variants that conform to the same structure.
Finally, Section~5 applies our results to \textsc{Yashima} game on bipartite graphs,
establishing that every position value is an integer pair.

\section{Preliminaries: Combinatorial Game Theory}

We assume basic familiarity with combinatorial game theory
and recall only the notions and results needed for this paper;
for background, see \cite{ANW19,S13}.
By convention, the two players are called Left and Right.

\subsection{Conway Algebra}

We introduce the Conway algebra,
which provides a theoretical basis for combinatorial game theory.

The underlying set $\mathbb{G}$ consists of \emph{positions}.
A \emph{position} $G$ ($\in \mathbb{G}$) is inductively defined as
an ordered pair $\cgpos{G^\mathcal{L}}{G^\mathcal{R}}$,
where $G^{\mathcal L}$ (resp.\ $G^{\mathcal R}$) is a finite set of positions
called the \emph{Left} (resp.\ \emph{Right}) \emph{options} of $G$.
We denote by $0$ the simplest position $\cgpos{}{}$.
When no confusion arises,
we use $G^{\mathcal L}$ and $G^{\mathcal R}$ to denote these option sets,
and we write $G^L$ (resp.\ $G^R$)
for a particular Left (resp.\ Right) option of $G$.
When Left (resp.\ Right) moves from $G$ to $G'$ in one move,
we write $G \move{L}{} G'$ (resp.\ $G \move{R}{} G'$),
or simply write $G \to G'$.
We say that two positions $G$ and $H$ are \emph{isomorphic},
written $G \cong H$, if their game trees coincide.

We partition the set of positions into
the four \emph{outcome classes} 
$\mathcal{N}$, $\mathcal{P}$, $\mathcal{L}$, and $\mathcal{R}$,
which are inductively defined as follows:
\begin{center}
\begin{tabular}{lcl}
  $G \in \mathcal{L}$ &$\iffsmall$&
    $\exists G^L, G^L \in \mathcal{L} \cup \mathcal{P}$
    and $\forall G^R, G^R \in \mathcal{L} \cup \mathcal{N}$ \\
  $G \in \mathcal{R}$ &$\iffsmall$&
    $\forall G^L, G^L \in \mathcal{R} \cup \mathcal{N}$
    and $\exists G^R, G^R \in \mathcal{R} \cup \mathcal{P}$ \\
  $G \in \mathcal{P}$ &$\iffsmall$&
    $\forall G^L, G^L \in \mathcal{R} \cup \mathcal{N}$
    and $\forall G^R, G^R \in \mathcal{L} \cup \mathcal{N}$ \\
  $G \in \mathcal{N}$ &$\iffsmall$&
    $\exists G^L, G^L \in \mathcal{L} \cup \mathcal{P}$
    and $\exists G^R, G^R \in \mathcal{R} \cup \mathcal{P}$
\end{tabular}
\end{center}
We denote by $o(G)$ the outcome class of $G$.
Intuitively,
$G \in \mathcal{L}$
(resp. $G \in \mathcal{R}$, $G \in \mathcal{P}$, and $G \in \mathcal{N}$)
means that
the Left (resp. Right, Previous, Next) player can win.


The negative $-G$ and the \emph{disjunctive sum} $G + H$ are inductively defined as
$\combgame{\{ -G^\mathcal{R} | -G^\mathcal{L} \}}$ and
$\combgame{\{ G^\mathcal{L}+H, G+H^\mathcal{L} | G^\mathcal{R}+H, G+H^\mathcal{R} \}}$.
The equivalence $G = H$ is defined as $\forall X, o(G+X) = o(H+X)$.
The algebra also has a partial order $\leq$ on positions
whose equivalence part ${\leq} \cap {\geq}$ is equal to $=$
and it satisfies the following properties:
\begin{center}
\begin{tabular}{lclclcl}
  $G > 0$ &$\iffsmall$& $G \in \mathcal{L}$
  &&
  $G < 0$ &$\iffsmall$& $G \in \mathcal{R}$ \\
  $G = 0$ &$\iffsmall$& $G \in \mathcal{P}$
  &&
  $G \cgfuzzy 0$ &$\iffsmall$& $G \in \mathcal{N}$
\end{tabular}
\end{center}
where $\cgfuzzy$ denotes the incomparability (``fuzzy'') relation.
We also use the derived relations $\geq$, $\leq$, $\cggfuz$, and $\cglfuz$
defined respectively as ${>} \cup {=}$, ${<} \cup {=}$, ${>} \cup {\cgfuzzy}$,
and ${<} \cup {\cgfuzzy}$.
We can extend these operations and relations to the quotient set $\mathbb{G}/{=}$:
$[G]+[H] = [G+H]$, $-[G] = [-G]$, and $[G] \leq [H] \Leftrightarrow G \leq H$.
The Conway algebra is the triple $\langle \mathbb{G}, +, \leq \rangle$,
whose quotient algebra $\langle \mathbb{G}/{=}, +, \leq \rangle$
is known to be an ordered Abelian group.
For convenience, 
the quotient algebra is also sometimes called the Conway algebra.

A Left option $G^L$ is said to be \emph{dominated}
if there exists another Left option $G^{L'}$ with $G^L \le G^{L'}$;
dually, a Right option $G^R$ is \emph{dominated}
if there exists $G^{R'}$ with $G^R \ge G^{R'}$.
It is known that
deleting a dominated option preserves equality, that is:
\[\hspace*{-10pt}\begin{array}{ll}
  \cgpos{G^\mathcal{L}}{G^\mathcal{R}}
    = \cgpos{G^\mathcal{L} \setminus \{G^L\}}{G^\mathcal{R}}
    & \text{if $G^L$ is dominated} \\
  \cgpos{G^\mathcal{L}}{G^\mathcal{R}}
    = \cgpos{G^\mathcal{L}}{G^\mathcal{R} \setminus \{G^R\}}
    & \text{if $G^R$ is dominated}
\end{array}\]
Another standard reduction is \emph{bypassing} \emph{reversible options},
which also preserves equality.
By repeatedly removing dominated options and bypassing reversible options,
one obtains the \emph{canonical form} of $G$.
Canonical forms are unique up to equality
and thus serve as canonical representative of each equality class.

The algebra $\mathbb{G}/{=}$ contains ordered Abelian subgroups
isomorphic to the integers and the dyadic rationals
with the usual addition and order.
We also denote by $\mathbb{Z}$ (resp. $\mathbb{D}$)
the set of the positions
in the equivalence classes corresponding to the integers
(resp. the dyadic rationals).
For a dyadic rational $x$,
we also write $x$ for the canonical form of its correspondence equivalence class.
The correspondence is as follows:
$0 \cong \cgpos{}{}$,
$m \cong \cgpos{m-1}{}$,
$-m \cong \cgpos{}{1-m}$, and
$\frac{n}{2^d} \cong \cgpos{\frac{n-1}{2^d}}{\frac{n+1}{2^d}}$,
where $m$ and $d$ are positive integers and $n$ is an odd integer.
In addition, there are many non-numeric positions, such as
$\cgstar \cong \cgpos{0}{0}$, $\cgup \cong \cgpos{0}{\cgstar}$,
and $\cgup\cgstar \cong \cgpos{0,\cgstar}{0}$.

We now introduce the notions of stops,
which give rise to one of the representative features.

\begin{definition}
Let $\mathcal{A} \in \{\Z,\D\}$.
For a position $G$, we define $\LS{\mathcal{A}}(G)$ and $\RS{\mathcal{A}}(G)$ as follows.
\begin{eqnarray*}
  \LS{\mathcal{A}}(G) &=& \left\{\begin{array}{ll}
    G & \mbox{if $G \in \mathcal{A}$} \\
    \max\{ \RS{\mathcal{A}}(G^L) : G^L \in G^\mathcal{L} \}
      & \mbox{otherwise}
  \end{array}\right.\\
  \RS{\mathcal{A}}(G) &=& \left\{\begin{array}{ll}
    G & \mbox{if $G \in \mathcal{A}$} \\
    \min\{ \LS{\mathcal{A}}(G^R) : G^R \in G^\mathcal{R} \}
      & \mbox{otherwise}
  \end{array}\right.
\end{eqnarray*}
We call $\LS{\D}(G)$ (resp.\ $\RS{\D}(G)$) the \emph{left} (resp.\ \emph{right}) \emph{stops},
and $\LS{\Z}(G)$ (resp.\ $\RS{\Z}(G)$) the \emph{integer left} (resp.\ \emph{right}) \emph{stops}.
\end{definition}
The stops are known to be invariant: that is,
if $G=H$ then
$\LS{\mathcal{A}}(G)=\LS{\mathcal{A}}(H)$
and $\RS{\mathcal{A}}(G)=\RS{\mathcal{A}}(H)$
for $\mathcal{A}\in\{\Z,\D\}$.
The stops also has the following proposition.

\begin{proposition}\label{Pro:stop}
Let $\mathcal{A}\in\{\Z,\D\}$, $G$ be a position, and $x\in \mathcal{A}$.
Then $\LS{\mathcal{A}}(G) < x \imply G < x$
and $\RS{\mathcal{A}}(G) > x \imply G > x$.
\end{proposition}

Although the case of $\mathcal{A} = \Z$ in the above proposition
is not directly shown in \cite{S13},
they can be shown in a similar way to the case of $\mathcal{A} = \D$.

We introduce two basic theorems in combinatorial game theory.

\begin{proposition}[Translation Theorem]\label{pro:trans}
Let $\mathcal{A} \in \{\Z,\D\}$, $x \in \mathcal{A}$, and $G \notin \mathcal{A}$.
Then $G+x = \cgpos{G^{\mathcal{L}}+x}{G^{\mathcal{R}}+x}$.
\end{proposition}

\begin{proposition}[Simplicity Theorem]\label{pro:simpl}
Let $\mathcal{A} \in \{\Z,\D\}$ and $G$ be a position.
If there exists $x \in \mathcal{A}$ such that
$G^{\mathcal{L}} \cglfuz x \cglfuz G^{\mathcal{R}}$,
then $G \in \mathcal{A}$.
More specifically,
$G$ is the simplest position among
$\{\,y \in \mathcal{A} :\; G^{\mathcal L} \cglfuz y \cglfuz G^{\mathcal R}\,\}$.
\end{proposition}

The notion ``simplest'' in the simplicity theorem
is defined in terms of \emph{birthday}.
Intuitively, the birthday of $G$ is the height of the game tree of $G$.

The following proposition follows from the simplicity theorem.

\begin{proposition}\label{pro:empty-int}
If $G^\mathcal{L} = \emptyset \lor G^\mathcal{R} = \emptyset$
then $G \in \mathbb{Z}$.
\end{proposition}

We introduce two technical propositions.

\begin{proposition}
$G^L \cglfuz G \cglfuz G^R$
for any $G$, $G^L$, and $G^R$.
\end{proposition}

\begin{proposition}\label{pro:transitive}
For any positions $G,H,J$,
if $G \cglfuz H \leq J$ or $G \leq H \cglfuz J$,
then $G \cglfuz J$.
\end{proposition}

The proof of Proposition \ref{pro:transitive}
is not explicitly written in \cite{ANW19,S13},
so we will provide it here.

\begin{proof}
Assume that $G \geq J$.
Then we have $H \leq G$ or $J \leq H$
because $H \leq J \leq G$ or $J \leq G \leq H$, respectively.
It is a contradiction
with $G \cglfuz H$ or $H \cglfuz J$, respectively.
\QED
\end{proof}

Finally, we introduce the notions of
\emph{integer pairs} $\ZZ$ and \emph{number pairs} $\DD$,
which play an important role in this paper.
\begin{eqnarray*}
  \ZZ &=& \bigl\{\,\cgpos{m}{n}\;:\; m,n \in \Z \,\bigr\} \\
  \DD &=& \bigl\{\,\cgpos{x}{y}\;:\; x,y \in \D \,\bigr\}
\end{eqnarray*}

By the simplicity theorem,
it is easy to see that
$n = \cgpos{n-1}{n+1}$ for any $n \in \Z$.
Since $\frac{n}{2^d} \cong \cgpos{\frac{n-1}{2^d}}{\frac{n+1}{2^d}}$,
we obtain the following proposition.

\begin{proposition}
For $\mathcal{A}\in\{\Z,\D\}$,
we have $\mathcal{A} \subset \cgpos{\mathcal{A}}{\mathcal{A}}$.
\end{proposition}

\begin{proposition}\label{pro:apair-a}
Let $\mathcal{A}\in\{\Z,\D\}$ and $x_1,x_2,y\in \mathcal{A}$ with
$\cgpos{x_1}{x_2} \in \cgpos{\mathcal{A}}{\mathcal{A}}\setminus \mathcal{A}$.
Then $\cgpos{x_1}{x_2} \cglfuz y \Leftrightarrow x_2 \leq y$.
\end{proposition}

\begin{proof}
From Proposition \ref{pro:trans},
we have $\cgpos{x_1}{x_2} - y = \cgpos{x_1-y}{x_2-y}$.
Hence
\(\cgpos{x_1}{x_2} \cglfuz y
  \Leftrightarrow \cgpos{x_1}{x_2} - y \cglfuz 0
  \Leftrightarrow \cgpos{x_1-y}{x_2-y} \cglfuz 0
  \Leftrightarrow \cgpos{x_1-y}{x_2-y} \in \mathcal{R} \cup \mathcal{N}
  \Leftrightarrow x_2-y \in \mathcal{R} \cup \mathcal{P}
  \Leftrightarrow x_2-y \leq 0
  \Leftrightarrow x_2 \leq y\).
\QED
\end{proof}

\subsection{Ruleset}

We introduce the concept of a ruleset
as a formalization of the combinatorial games studied in this paper
(we restrict our attention to short partizan games under normal play).

\begin{definition}
A \emph{ruleset} is a triple
$\langle \mathcal{G}, {\move{L}{}}, {\move{R}{}} \rangle$,
where
$\mathcal{G}$ is the set of \emph{game positions},
and both ${\move{L}{}}$ and ${\move{R}{}}$
are well-founded and finitely branching relations on $\mathcal{G}$.
\end{definition}

Here, $G \move{L}{} G'$ (resp. $G \move{R}{} G'$)
represents a move by the Left (resp. Right) player
in the given combinatorial game.
We adopt the normal play convention,
under which the player who makes the last move is the winner.

For a game position $G$
with $\{G^{L_1},\ldots,G^{L_m}\} = \{ G' \mid G \move{L}{} G'\}$
and $\{G^{R_1},\ldots,G^{R_n}\} = \{ G' \mid G \move{R}{} G'\}$,
the game position $G$ is assigned a position
$H \cong \combgame{\{H^{L_1},\ldots,H^{L_m} \mid H^{R_1},\ldots,H^{R_n}\}}$
in the Conway algebra,
where $H^{L_i}$ and $H^{R_j}$ are the assigned positions of
$G^{L_i}$ and $G^{R_j}$, respectively.
According to convention,
we identify a game position with the assigned position,
and import various notions from the Conway algebra into game positions.
We call the canonical form of the assigned position of $G$
the \emph{position value} (the \emph{game value}, simply the \emph{value}) of $G$.

\section{Basic Diamond Properties}

In this section, we introduce the concept of
the $\Diamond_\Z$-property and the $\Diamond_\D$-property,
and explain how it relates to the integers $\Z$, the numbers $\D$, 
the integer pairs $\ZZ$, and the number pairs $\DD$.

\begin{lemma}\label{lem:diamond1}
Let $\mathcal{A} \in \{\Z, \D\}$, 
$G_0 \in \cgpos{\mathcal{A}}{\mathcal{A}}$ and $G_1 \in \cgpos{\mathcal{A}}{\mathcal{A}} \setminus \mathcal{A}$.
Then the following properties hold for any $x \in \mathcal{A}$:
\begin{itemize}
\item
  $\RS{\mathcal{A}}(G_1) \leq \RS{\mathcal{A}}(G_0)$ and $G_0 \cglfuz x$ $\imply G_1 \cglfuz x$
\item
  $\LS{\mathcal{A}}(G_1) \geq \LS{\mathcal{A}}(G_0)$ and $G_0 \cggfuz x$ $\imply G_1 \cggfuz x$
\end{itemize}
\end{lemma}

\begin{proof}
We will prove only the first claim.
The second can be shown in a similar way.
Let $G_1 = \cgpos{z_1}{z_2}$.

In the case of $G_0 = y \in \mathcal{A}$:
We have $z_2 \leq y$ and $y < x$
from $\RS{\mathcal{A}}(G_1) \leq \RS{\mathcal{A}}(G_0)$ and $G_0 \cglfuz x$,
respectively.
It follows from Proposition \ref{pro:apair-a}
that $G_1 = \cgpos{z_1}{z_2} \cglfuz y < x$.
Hence, by Proposition \ref{pro:transitive},
we obtain $G_1 \cglfuz x$.

In the case of $G_0 = \cgpos{y_1}{y_2} \in \cgpos{\mathcal{A}}{\mathcal{A}} \setminus \mathcal{A}$:
We have $z_2 \leq y_2$ and $y_2 \leq x$
from $\RS{\mathcal{A}}(G_1) \leq \RS{\mathcal{A}}(G_0)$ and $G_0 \cglfuz x$,
respectively.
It follows from Proposition \ref{pro:apair-a}
that $G_1 = \cgpos{z_1}{z_2} \cglfuz y_2 \leq x$.
Hence, by Proposition \ref{pro:transitive}, we obtain $G_1 \cglfuz x$.
\QED
\end{proof}

The property given in this lemma does not hold in general.
For instance, it fails if we allow $G_1 \in \mathcal{A}$.
A concrete counter example is
$G_0 = \cgstar$ and $G_1 = 0$
because $\RS{\mathcal{A}}(G_1) = 0 = \RS{\mathcal{A}}(G_0)$
and $G_0 \cglfuz 0$ but $G_1 = 0 \not \cglfuz 0$.

Within the framework of this paper,
positions that are fuzzy to each other --- such as $0$ and $\cgstar$ ---
do not appear simultaneously among the Left (or Right) options,
as in this example.
This restriction prevents canonical forms from becoming quite complicated.

\begin{definition}
Let $\mathcal{A} \in \{\Z,\D\}$.
For a position $G$,
the sets of \emph{guide left  options} and \emph{guide right options}, denoted by $\gd_\mathcal{A}^L(G)$ and $\gd_\mathcal{A}^R(G)$, are defined as follows:
\[\begin{array}{rcl}
  \gd_\mathcal{A}^L(G) &=&
    \left\{\begin{array}{ll}
      \emptyset
        & \text{ if $G \in \mathcal{A}$} \\
      \{G^L : G^L = \LS{\mathcal{A}}(G)\}
        & \text{ if $G \notin \mathcal{A}$ and $\LS{\mathcal{A}}(G) \in G^\mathcal{L}$} \\
      \{G^L : \RS{\mathcal{A}}(G^L) = \LS{\mathcal{A}}(G)\}
        &\text{ otherwise}
    \end{array}\right. \\[20pt]
  \gd_\mathcal{A}^R(G) &=&
    \left\{\begin{array}{ll}
      \emptyset
        & \text{ if $G \in \mathcal{A}$} \\
      \{G^R : G^R = \RS{\mathcal{A}}(G)\}
        & \text{ if $G \notin \mathcal{A}$ and $\RS{\mathcal{A}}(G) \in G^\mathcal{R}$} \\
      \{G^R : \LS{\mathcal{A}}(G^R) = \RS{\mathcal{A}}(G)\}
        &\text{ otherwise}
    \end{array}\right. 
\end{array}\]
\end{definition}

\begin{definition}\label{def:diamond1}
Let $\mathcal{A} \in \{\Z,\D\}$. We say that a position $G$ has
the \emph{diamond property for $\mathcal{A}$
or simply} \emph{$\Diamond_\mathcal{A}$-property}
if there exist $G^L \in \gd_\mathcal{A}^L(G)$, $G^R \in \gd_\mathcal{A}^R(G)$, and $x \in \mathcal{A}$
such that
$G^L \cglfuz x \cglfuz G^R$.
\end{definition}

\begin{figure}[htb]
\centering
\begin{tikzpicture}
\node at (0,6) {
\begin{tikzpicture}[x=14pt,y=14pt]
\node at (2.2,1.9) {$\Diamond_\Z$};
\node (g) at (2,4) {$G$};
\node (gl) at (0,2) {$G^L$};
\node (gr) at (4,2) {$G^R$};
\node (h) at (3,0) {$x~(\in \Z)$};
\draw [->,thick] (g) to (gl);
\draw [->,thick] (g) to (gr);

\coordinate (lfl) at ($(0,1.7)+(-45:.8)$);
\draw [thin] (lfl) -- ($(lfl)+(-25:.6)$) -- ($(lfl)+(-65:.6)$) --cycle;
\draw [thin] ($(lfl)+(-25:.6)+(-45:.15)$) -- ($(lfl)+(-65:.6)+(-45:.15)$);

\coordinate (lfr) at ($(2.3,1.1)+(-45:.8)$);
\draw [thin] (lfr) -- ($(lfr)+(65:.6)$) -- ($(lfr)+(25:.6)$) --cycle;
\draw [thin] ($(lfr)+(65:.6)+(45:.15)$) -- ($(lfr)+(25:.6)+(45:.15)$);
\end{tikzpicture}};
\node at (4,6) {
\begin{tikzpicture}[x=14pt,y=14pt]
\node at (2.2,1.9) {$\Diamond_\D$};
\node (g) at (2,4) {$G$};
\node (gl) at (0,2) {$G^L$};
\node (gr) at (4,2) {$G^R$};
\node (h) at (3,0) {$x~(\in \D)$};
\draw [->,thick] (g) to (gl);
\draw [->,thick] (g) to (gr);

\coordinate (lfl) at ($(0,1.7)+(-45:.8)$);
\draw [thin] (lfl) -- ($(lfl)+(-25:.6)$) -- ($(lfl)+(-65:.6)$) --cycle;
\draw [thin] ($(lfl)+(-25:.6)+(-45:.15)$) -- ($(lfl)+(-65:.6)+(-45:.15)$);

\coordinate (lfr) at ($(2.3,1.1)+(-45:.8)$);
\draw [thin] (lfr) -- ($(lfr)+(65:.6)$) -- ($(lfr)+(25:.6)$) --cycle;
\draw [thin] ($(lfr)+(65:.6)+(45:.15)$) -- ($(lfr)+(25:.6)+(45:.15)$);
\end{tikzpicture}};
\end{tikzpicture}
\caption{$\Diamond_\Z$-property and $\Diamond_\D$-property}
\end{figure}
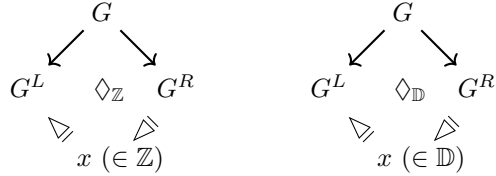

\begin{lemma}\label{lem:diamond2}
Let $\mathcal{A} \in \{\Z, \D\}$.
If a position $G$ has the $\Diamond_\mathcal{A}$-property
and $G^\mathcal{L}, G^\mathcal{R} \subset \cgpos{\mathcal{A}}{\mathcal{A}}$,
then $G \in \mathcal{A}$.
\end{lemma}

\begin{proof}
Assume that $G \notin \mathcal{A}$.
From the $\Diamond_\mathcal{A}$-property,
there exist a guide left  option $G^L$,
a guide right option $G^R$, and
$x \in \mathcal{A}$ such that
$G^L \cglfuz x \cglfuz G^R$.

We first show $G^{L'} \cglfuz x$ for any $G^{L'}\in G^{\mathcal{L}}\setminus \{G^L\}$.
Since $G^L$ is a guide left option, we have $\RS{\mathcal{A}}(G^{L'})\leq \RS{\mathcal{A}}(G^L)$.
\begin{itemize}
 \item
   In the case of $G^{L'} \notin \mathcal{A}$:
   Then $G^{L'} \cglfuz x$ follows from Lemma \ref{lem:diamond1}.
\item
  In the case of $G^{L'} \in \mathcal{A}$ and $\RS{\mathcal{A}}(G^{L'})<\RS{\mathcal{A}}(G^L)$:
  Then $G^{L'} < G^L$ follows from Proposition \ref{Pro:stop},
  and hence $G^{L'} \cglfuz x$.
\item
  In the case of $G^{L'} \in \mathcal{A}$ and $\RS{\mathcal{A}}(G^{L'})=\RS{\mathcal{A}}(G^L)$:
  Since $G^L$ is a guide left option, we have $G^L \in \mathcal{A}$.
  Hence $G^{L'} = G^L$, and $G^{L'} \cglfuz x$.
\end{itemize}
Thus, we have $G^{\mathcal{L}} \cglfuz x$. Similarly, we can also show $x \cglfuz G^{\mathcal{R}}$.
Therefore, Proposition \ref{pro:simpl} derives $G \in \mathcal{A}$,
which is a contradiction.
\QED
\end{proof}

We are now ready to state the theorem
that forms the foundation of all the results in this paper.

\begin{theorem}\label{th:diamond3}
Suppose that a set $\mathcal{G}$ of positions is closed under options.
For each $\mathcal{A} \in \{\Z,\D\}$,
if $\mathcal{G}$ is partitioned into $\mathcal{D}$ and $\mathcal{S}$
such that:
\begin{enumerate}
\item[(a)]
  For any $G \in \mathcal{D}$,
  either $G \in \mathcal{A}$ is known
  or $G$ has the $\Diamond_\mathcal{A}$-property, and
\item[(b)]
  $G \in \mathcal{S} \imply G^\mathcal{L}, G^\mathcal{R} \subset \mathcal{D}$,
\end{enumerate}
then the following properties hold:
\begin{enumerate}
\item[(1)]
  $G \in \mathcal{G} \imply G \in \cgpos{\mathcal{A}}{\mathcal{A}}$
\item[(2)]
  $G \in \mathcal{D} \imply G \in \mathcal{A}$
\end{enumerate}
\end{theorem}

\begin{proof}
We proceed by induction on $G$ ordered by the moving relation $\to$.
\begin{itemize}
\item
  In the case of $G \in \mathcal{A}$:
  It is trivial
  because of $\mathcal{A} \subset \cgpos{\mathcal{A}}{\mathcal{A}}$.
\item
  In the case of $G \notin \mathcal{A}$ and $G \in \mathcal{S}$:
  From the assumption (b)
  and the induction hypothesis (2),
  we have $G^\mathcal{L}, G^\mathcal{R} \subset \mathcal{A}$.
  In either case $\mathcal{A} = \Z$ or $\mathcal{A} = \D$, any two positions in $\mathcal{A}$ can be compared.
  Hence, in $G^\mathcal{L}$ and $G^\mathcal{R}$, we can remove all dominated options and keep at most one on each side.
  Therefore, $G\in\cgpos{\mathcal{A}}{\mathcal{A}}$.
\item
  In the case of $G \notin \mathcal{A}$ and $G \in \mathcal{D}$:
  From the induction hypothesis (1), we have $G^\mathcal{L}, G^\mathcal{R} \subset \cgpos{\mathcal{A}}{\mathcal{A}}$.
  Since $G$ has the $\Diamond_\mathcal{A}$-property, Lemma \ref{lem:diamond2} implies that $G \in \mathcal{A}$, which is a contradiction.
  Therefore, this case does not occur.
\QED
\end{itemize}
\end{proof}

The phrase ``$G \in \mathcal{A}$ is known'' in condition (a) of this theorem
allows the use of arbitrary methods
for determining that a position is in $\mathcal{A}$
(e.g., Proposition \ref{pro:empty-int}).

By considering the special case $\mathcal{D}=\mathcal{G}$ and $\mathcal{S}=\emptyset$, we obtain the following corollary.

\begin{corollary}\label{coro:diamond4}
Suppose that a set $\mathcal{G}$ of positions is closed under options
and $\mathcal{A} \in \{\Z,\D\}$.
If for any $G \in \mathcal{G}$,
either $G \in \mathcal{A}$ is known
or $G$ has the $\Diamond_\mathcal{A}$-property,
then $G \in \mathcal{A}$.
\end{corollary}

\section{Various Diamond Properties}

In this section,
we introduce various diamond properties.
We begin with those that are instances of 
both the $\Diamond_\Z$-property and the $\Diamond_\D$-property.

\begin{definition}\label{def:dia-Zs}
We say that a position $G$ has:
\begin{itemize}
\item
  the \emph{$\Diamond$-property} for $\mathcal{A}$
  if 
  there exist 
  $G^L \in \gd_{\mathcal{A}}^L(G)$, $G^R \in \gd_{\mathcal{A}}^R(G)$,
  and $H$
  such that $G^L \move{R}{} H \lmove{L}{} G^R$,
\item
  the \emph{$\Diamond^{\leq}$-property} for $\mathcal{A}$
  if 
  there exist 
  $G^L \in \gd_{\mathcal{A}}^L(G)$, $G^R \in \gd_{\mathcal{A}}^R(G)$,
  $G^{LR}$, and $G^{RL}$
  such that $G^{LR} \leq G^{RL}$,
\item
  the \emph{$\Diamond_L^{\cglfuz}$-property} for $\mathcal{A}$
  if 
  there exist 
  $G^L \in \gd_{\mathcal{A}}^L(G)$, $G^R \in \gd_{\mathcal{A}}^R(G)$,
  and $G^{LR}$
  such that $G^{LR} \cglfuz G^R$, and
\item
  the \emph{$\Diamond_R^{\cglfuz}$-property} for $\mathcal{A}$
  if 
  there exist 
  $G^L \in \gd_{\mathcal{A}}^L(G)$, $G^R \in \gd_{\mathcal{A}}^R(G)$,
  and $G^{RL}$
  such that $G^L \cglfuz G^{RL}$.
\end{itemize}
\end{definition}

We display these diamond properties in Fig.\ref{fig:dia-Zs}.
Note that,
when combined with Proposition \ref{pro:empty-int},
the $\Diamond^{\leq}$-property corresponds to the $F_2$-property
introduced in \cite{CHNS21},
and includes the $\Diamond$-property as the special case
where $G^{LR}$ and $G^{RL}$ are the same position.

\begin{figure}[htb]
\centering
\begin{tikzpicture}
\node at (0,0) {
\begin{tikzpicture}[x=14pt,y=14pt]
\node at (2,2) {$\Diamond$};
\node (g) at (2,4) {$G$};
\node (gl) at (0,2) {$G^L$};
\node (gr) at (4,2) {$G^R$};
\node (h) at (2,0) {$H$};
\draw [->,thick] (g) to (gl);
\draw [->,thick] (g) to (gr);
\draw [->,thick,dotted] (gl) to (h);
\draw [->,thick,dotted] (gr) to (h);
\end{tikzpicture}};
\node at (3.5,0) {
\begin{tikzpicture}[x=14pt,y=14pt]
\node at (2.1,1.9) {$\Diamond^{\leq}$};
\node (g) at (2,4) {$G$};
\node (gl) at (0,2) {$G^L$};
\node (gr) at (4,2) {$G^R$};
\node (glr) at (.7,0) {$G^{LR}$};
\node (grl) at (3.4,0) {$G^{RL}$};
\node at (2,0) {$\leq$};
\draw [->,thick] (g) to (gl);
\draw [->,thick] (g) to (gr);
\draw [->,thick,dotted] (gl) to (glr);
\draw [->,thick,dotted] (gr) to (grl);
\end{tikzpicture}};
\node at (7,0) {
\begin{tikzpicture}[x=14pt,y=14pt]
\node at (2,1.9) {$\Diamond_L^{\cglfuz}$};
\node (g) at (2,4) {$G$};
\node (gl) at (0,2) {$G^L$};
\node (gr) at (4,2) {$G^R$};
\node (glr) at (2,0) {$G^{LR}$};
\draw [->,thick] (g) to (gl);
\draw [->,thick] (g) to (gr);
\draw [->,thick,dotted] (gl) to (glr);
\coordinate (lfr) at ($(2.3,1.1)+(-45:.8)$);
\draw [thin] (lfr) -- ($(lfr)+(65:.6)$) -- ($(lfr)+(25:.6)$) --cycle;
\draw [thin] ($(lfr)+(65:.6)+(45:.15)$) -- ($(lfr)+(25:.6)+(45:.15)$);
\end{tikzpicture}};
\node at (10.5,0) {
\begin{tikzpicture}[x=14pt,y=14pt]
\node at (2,1.9) {$\Diamond_R^{\cglfuz}$};
\node (g) at (2,4) {$G$};
\node (gl) at (0,2) {$G^L$};
\node (gr) at (4,2) {$G^R$};
\node (grl) at (2,0) {$G^{RL}$};
\draw [->,thick] (g) to (gl);
\draw [->,thick] (g) to (gr);
\draw [->,thick,dotted] (gr) to (grl);
\coordinate (lfl) at ($(0,1.7)+(-45:.8)$);
\draw [thin] (lfl) -- ($(lfl)+(-25:.6)$) -- ($(lfl)+(-65:.6)$) --cycle;
\draw [thin] ($(lfl)+(-25:.6)+(-45:.15)$) -- ($(lfl)+(-65:.6)+(-45:.15)$);
\end{tikzpicture}};
\end{tikzpicture}
\caption{Diamond properties in Definition \ref{def:dia-Zs}}
\label{fig:dia-Zs}
\end{figure}
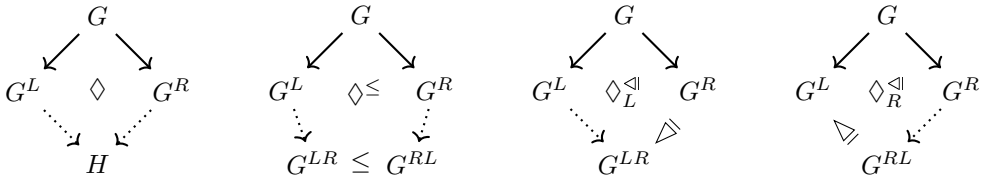

In both cases $\mathcal{A} \in \{\Z,\D\}$,
by adding the assumption
\begin{enumerate}
\item[(c)]
  $G \in \mathcal{D} \imply G^\mathcal{LR}, G^\mathcal{RL} \subset \mathcal{D}$
\end{enumerate}
to Theorem \ref{th:diamond3},
we can assume that 
$G^\mathcal{LR}, G^\mathcal{RL} \subset \mathcal{A}$
as an induction hypothesis in the proof.
Under the hypothesis,
every diamond property in Definition \ref{def:dia-Zs}
guarantees the $\Diamond_\mathcal{A}$-property for each $\mathcal{A} \in \{\Z,\D\}$.
\begin{description}
\item[$\Diamond$-property:]
  We have $H \in \mathcal{A}$ and $G^L \cglfuz H \cglfuz G^R$
  because of $G^L \move{R}{} H \lmove{L}{} G^R$.
\item[$\Diamond^\leq$-property:]
  We have $G^{LR}, G^{RL} \in \mathcal{A}$ and $G^L \cglfuz G^{LR},G^{RL} \cglfuz G^R$
  because of $G^L \move{R}{} G^{LR} \leq G^{RL} \lmove{L}{} G^R$.
\item[$\Diamond_L^{\cglfuz}$-property:]
  We have $G^{LR} \in \mathcal{A}$ and $G^L \cglfuz G^{LR} \cglfuz G^R$
  because of $G^L \move{R}{} G^{LR} \cglfuz G^R$.
\item[$\Diamond_R^{\cglfuz}$-property:]
  We have $G^{RL} \in \mathcal{A}$ and $G^L \cglfuz G^{RL} \cglfuz G^R$
  because of $G^L \cglfuz G^{RL} \lmove{L}{} G^R$.
\end{description}
Hence we obtain the following corollary of Theorem \ref{th:diamond3}.

\begin{corollary}\label{co:variantZZ}
Suppose that a set $\mathcal{G}$ of positions is closed under options.
If $\mathcal{G}$ is partitioned into $\mathcal{D}$ and $\mathcal{S}$
such that:
\begin{enumerate}
\item[(a)]
  For any $G \in \mathcal{D}$,
  either $G \in \mathbb{Z}$ is known
  or
  $G$
  has one of diamond properties for $\Z$ in Definition \ref{def:dia-Zs},
\item[(b)]
  $G \in \mathcal{S} \imply G^\mathcal{L}, G^\mathcal{R} \subset \mathcal{D}$, and
\item[(c)]
  $G \in \mathcal{D} \imply G^\mathcal{LR}, G^\mathcal{RL} \subset \mathcal{D}$,
\end{enumerate}
then the following properties hold:
\begin{enumerate}
\item[(1)]
  $G \in \mathcal{G} \imply G \in \ZZ$
\item[(2)]
  $G \in \mathcal{D} \imply G \in \Z$
\end{enumerate}
\end{corollary}

We also obtain the following corollary
as the special case $\mathcal{D} = \mathcal{G}$ and $\mathcal{S} = \emptyset$.

\begin{corollary}\label{co:variantZ}
Suppose that a set $\mathcal{G}$ of positions is closed under options.
If for any $G \in \mathcal{G}$,
either $G \in \mathbb{Z}$ is known or
$G$
has one of diamond properties for $\Z$ in Definition \ref{def:dia-Zs},
then $\mathcal{G} \subset \Z$.
\end{corollary}

Note that,
when restricted to the $\Diamond^\leq$-property
and combined with Proposition \ref{pro:empty-int},
the corollary above 
essentially coincides with a result in \cite{CHNS21}.

Next, we introduce diamond properties
that are instances of the $\Diamond_\D$-property,
although they are not instances of the $\Diamond_\Z$-property.

\begin{definition}\label{def:dia-Ds}
We say that a position $G$ has:
\begin{itemize}
\item
  the \emph{$\Diamond_L^{\leq}$-property}
  if 
  there exist $G^L \in \gd_{\D}^L(G)$, $G^R \in \gd_{\D}^R(G)$,
  and $G^{LR}$
  such that $G^{LR} \leq G^R$,
\item
  the \emph{$\Diamond_R^{\leq}$-property}
  if 
  there exist $G^L \in \gd_{\D}^L(G)$, $G^R \in \gd_{\D}^R(G)$,
  and $G^{RL}$
  such that $G^L \leq G^{RL}$, and
\item
  the \emph{$\triangle$-property}
  if 
  there exist $G^L \in \gd_{\D}^L(G)$, $G^R \in \gd_{\D}^R(G)$,
  such that $G^L \cglfuz G^R$.
\end{itemize}
\end{definition}

We display these diamond properties in Fig.\ref{fig:dia-Ds}.
Note that the disjunction of the $\Diamond_L^\leq$ and $\Diamond_R^\leq$-properties
corresponds to the $F_1$-property introduced in \cite{CHNS21}.

\begin{figure}[htb]
\centering
\begin{tikzpicture}
\node at (0,0) {
\begin{tikzpicture}[x=14pt,y=14pt]
\node at (2,1.9) {$\Diamond_L^\leq$};
\node (g) at (2,4) {$G$};
\node (gl) at (0,2) {$G^L$};
\node (gr) at (4,2) {$G^R$};
\node (glr) at (2,0) {$G^{LR}$};
\draw [->,thick] (g) to (gl);
\draw [->,thick] (g) to (gr);
\draw [->,thick,dotted] (gl) to (glr);
\node at (3,.9) {\rotatebox{45}{$\leq$}};
\end{tikzpicture}};
\node at (3.5,0) {
\begin{tikzpicture}[x=14pt,y=14pt]
\node at (2,1.9) {$\Diamond_R^\leq$};
\node (g) at (2,4) {$G$};
\node (gl) at (0,2) {$G^L$};
\node (gr) at (4,2) {$G^R$};
\node (grl) at (2,0) {$G^{RL}$};
\draw [->,thick] (g) to (gl);
\draw [->,thick] (g) to (gr);
\draw [->,thick,dotted] (gr) to (grl);
\node at (.8,.9) {\rotatebox{-45}{$\leq$}};
\end{tikzpicture}};
\node at (7,0) {
\begin{tikzpicture}[x=14pt,y=14pt]
\node at (2,2.6) {$\triangle$};
\node (g) at (2,4) {$G$};
\node (gl) at (0,1.5) {$G^L$};
\node (gr) at (4,1.5) {$G^R$};
\draw [->,thick] (g) to (gl);
\draw [->,thick] (g) to (gr);
\node at (2,1.5) {$\cglfuz$};
\end{tikzpicture}};
\end{tikzpicture}
\caption{Diamond properties in Definition \ref{def:dia-Ds}}
\label{fig:dia-Ds}
\end{figure}
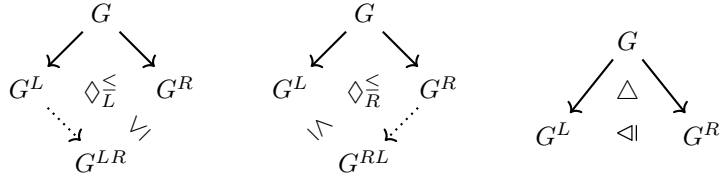

These three properties rely on the density of numbers.
They imply that $G^L \cglfuz G^R$.
Under the assumption that $G^L, G^R \in \D$,
we have $G^L < G^R$,
and hence, by the density of numbers,
there exists $x \in \D$ such that $G^L < x < G^R$.
On the other hand, 
these properties do not hold for integers,
because the density property failed for integers.
Moreover,
the requirement that $G^L, G^R \in \D$ even when $G \in \mathcal{D}$
implies that all positions in $\mathcal{G}$ must lie in $\mathcal{D}$.
Hence,
although we do not obtain the corollary
(similar to Corollary \ref{co:variantZZ})
of Theorem \ref{th:diamond3}
with respect to the diamond properties in Definition \ref{def:dia-Ds},
we do obtain the corollary stated in Corollary \ref{coro:diamond4}.

\begin{corollary}\label{co:variantD}
Suppose that a set $\mathcal{G}$ of positions is closed under options.
If for any $G \in \mathcal{G}$,
either $G \in \mathbb{D}$ is known or
$G$
has one of diamond properties for $\D$
in Definition \ref{def:dia-Zs} and Definition \ref{def:dia-Ds},
then $\mathcal{G} \subset \D$.
\end{corollary}

Note that,
when restricted to the $\Diamond^\leq$, $\Diamond_L^\leq$ and $\Diamond_R^\leq$-properties,
and combined with Proposition \ref{pro:empty-int},
the corollary above essentially coincides with a result in \cite{CHNS21}.

In the definitions of the various diamond properties,
we require the existence of guide options $G^L \in \gd_{\mathcal{A}}^L(G)$ and $G^R \in \gd_{\mathcal{A}}^R(G)$.
It follows from Proposition \ref{pro:empty-int}
that $G \in \Z \subset \D$ whenever $G^\mathcal{L} = \emptyset \lor G^\mathcal{R} = \emptyset$.
Hence, even if we modify the requirement so that ``all options'' satisfy the remaining conditions,
all the theorems and corollaries concerning the diamond properties will still hold.

\section{Case Study: \textsc{Yashima} Game on Bipartite Graphs}

\textsc{Yashima} game, proposed by Arisawa \cite{A90}, is played on a given undirected graph.
In this section, we study \textsc{Yashima} game on bipartite graphs as a combinatorial game
that satisfies the diamond property ($\Diamond$-property).
Using our results, we show that every game position in \textsc{Yashima} game on bipartite graphs takes a value of the very simple form $\cgpos{m}{n}\in \ZZ$.

To begin, recall that a graph is bipertite if and only if it is 2-colorable.
We classify game positions of \textsc{Yashima} game on  bipartite graphs as follows.
\begin{definition}
Let $G$ be a game position of \textsc{Yashima} game on  bipartite graphs.
\begin{itemize}
\item
  We call $G$ a \emph{different-color position}
  if there exists a 2-coloring in which the vertices of the Left and
	 Right tokens receive different colors.
\item
Otherwise, $G$ is called a \emph{same-color position}.
\end{itemize}
Let $\mathcal{D}$ (resp. $\mathcal{S}$) denote the set of all
different-color (resp. same-color) positions.
\end{definition}

With this partition, we obtain the following theorem.

\begin{theorem}\label{th:yashima-diamond}
\textsc{Yashima} game on bipartite graphs satisfies the $\Diamond$-property.
\end{theorem}

\begin{proof}
Let $G \in \mathcal{D}$ and suppose $G^L \lmove{L}{} G \move{R}{} G^R$.
We claim that the Left move used in $G \move{L}{} G^L$ can also be played from $G^R$,
yielding $G^R \move{L}{} G^{RL}$;
symmetrically, the Right move used in $G \move{R}{} G^R$ can also be played from $G^L$,
yielding $G^L \move{R}{} G^{LR}$.
Once this is verified, it follows that $G^{RL}\cong G^{LR}$,
and this common game position belongs to $\mathcal{D}$.
Hence $G$ has the $\Diamond$-property.

In \textsc{Yashima} game, a planned move can be blocked by the opponent
only in the following situation:
\begin{center}
\begin{tikzpicture}[x=14pt,y=14pt,thick]
\node (l) at (1,1) [draw,circle,inner sep=1pt] {L};
\node (p) at (3,1) [draw,circle,inner sep=4pt] {};
\node (r) at (5,1) [draw,circle,inner sep=1pt] {R};
\draw (l) -- (p) -- (r);
\draw [dotted] (-.8,.5) to (l);
\draw [dotted] (-.8,1.5) to (l);
\draw [dotted] (6.8,.5) to (r);
\draw [dotted] (6.8,1.5) to (r);
\end{tikzpicture}
\end{center}
That is, the opponent moves onto exactly the vertex  (the ``central'' vertex)
to which the player intended to move.
Such blocking is possible only in a same-color position.
Since $G \in \mathcal{D}$, 
no such blocking move is possible for either player.
\QED
\end{proof}

For the different-color positions $\mathcal{D}$
and the same-color positions $\mathcal{S}$,
conditions (b) and (c) of Corollary \ref{co:variantZZ}
follow directly from the rules of \textsc{Yashima} game.
Therefore, we obtain the following corollary.

\begin{corollary}\label{co:yashima}
Let $G$ be a game position of \textsc{Yashima} game on  bipartite graphs.
Then the following properties hold:
\begin{enumerate}
\item[(1)]
  $G\in \cgpos {\Z} {\Z}$.
\item[(2)]
  If $G$ is a different-color position, then $G\in\Z$.
\end{enumerate}
\end{corollary}

A game with rules similar to \textsc{Yashima} game is known as \textsc{Tron} \cite{HT90}.
In \textsc{Tron}, after moving a token,
one deletes the vertex just left,
rather than deleting an edge as in \textsc{Yashima} game.
Theorem \ref{th:yashima-diamond} and Corollary \ref{co:yashima} likewise hold for \textsc{Tron}
on bipartite graphs.

Finally, Fig.\ref{fig:yashima2} illustrates
a concrete example of a game position of \textsc{Yashima} game
on a bipartite graph.
The size of the fully expanded game tree for this game position reaches
104{,}241,
so the number of follower game positions is enormous.
For the analysis, one must also compute the canonical form
of this game position.
By using memoization,
we can reduce the search space to 1{,}307
(the exact number depends on the implementation).
Nevertheless, the computation typically require computer assistance.
In contrast,
the value $\cgpos{0}{{-}3}$ of this game position is remarkably simple.
Note that
the value of the same game position,
when regarded as a \textsc{Tron} game,
is $\cgpos{1}{{-}1}$, which is also simple.

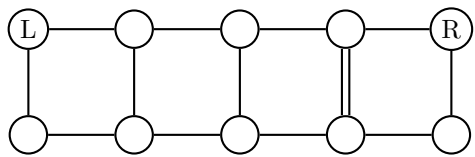
\begin{figure}[htb]
  \centering
\begin{tikzpicture}[x=20pt,y=20pt,thick]
\node (p01) at (0,2) [draw,circle,inner sep=2pt] {L};
\node (p11) at (2,2) [draw,circle,inner sep=5pt] {};
\node (p21) at (4,2) [draw,circle,inner sep=5pt] {};
\node (p31) at (6,2) [draw,circle,inner sep=5pt] {};
\node (p41) at (8,2) [draw,circle,inner sep=2pt] {R};
\node (p00) at (0,0) [draw,circle,inner sep=5pt] {};
\node (p10) at (2,0) [draw,circle,inner sep=5pt] {};
\node (p20) at (4,0) [draw,circle,inner sep=5pt] {};
\node (p30) at (6,0) [draw,circle,inner sep=5pt] {};
\node (p40) at (8,0) [draw,circle,inner sep=5pt] {};
\draw (p01) -- (p11) -- (p21) -- (p31) -- (p41);
\draw (p00) -- (p10) -- (p20) -- (p30) -- (p40);
\draw (p00) -- (p01);
\draw (p10) -- (p11);
\draw (p20) -- (p21);
\draw [double distance=2pt] (p30) -- (p31);
\draw (p40) -- (p41);
\end{tikzpicture}
  \caption{\textsc{Yashima} game (position value: $\{0 \mid -3\}$)}
  \label{fig:yashima2}
\end{figure}

\section*{Acknowledgements}
This work was partially supported by
JSPS KAKENHI Grant Number JP22K13953. 




\end{document}